\documentclass[letterpaper,12pt]{amsart}

\title{Normal Binary Graph Models}

\author{ Seth Sullivant}
\address{Department of Mathematics \\ North Carolina State University, Raleigh, NC 27695}
\email{smsulli2@ncsu.edu}

\date{}

\usepackage{latexsym,array,delarray,amsthm,amssymb,epsfig}%eufrak
\theoremstyle{plain}
\newtheorem{thm}{Theorem}[section]
\newtheorem{lemma}[thm]{Lemma}
\newtheorem{prop}[thm]{Proposition}

\theoremstyle{definition}

\newtheorem{ex}[thm]{Example}
\newtheorem{pr}[thm]{Problem}

\newtheorem{ques}[thm]{Question}

\theoremstyle{remark}

\newcommand{\bfi}{\mathbf{i}}

\newcommand{\ind}{\mbox{$\perp \kern-5.5pt \perp$}}

\begin{document}
\maketitle

\begin{abstract}
We show that the marginal semigroup of a binary graph model is normal if and only if the graph is free of $K_4$ minors.  The technique, based on the interplay of normality and the geometry of the marginal cone, has potential applications to other normality questions in algebraic statistics.
\end{abstract}

%%%%%%%%%%%%%%%%%%%%%%%%%%%%%%%%%%
%%%%%%%%%%%%%%%%%%%%%%%%%%%%%%%%%%
%%%%%%%%%%%%%%%%%%%%%%%%%%%%%%%%%%
%%%%%%%%%%%%%%%%%%%%%%%%%%%%%%%%%%

\section{Introduction}

The summary of high-dimensional data in a  multiway table by lower order marginals is a staple of the statistical sciences.  In the case where the table is a contingency table, that is, a table of counts, the table is a nonnegative integral multiway array, and the marginal summaries are a list of  lower order nonnegative integral multiway arrays.  For many applications, e.g. data confidentiality and hypothesis testing \cite{Diaconis1998}, we would like to determine whether a given list of lower order nonnegative integral arrays could actually be the list of margins of a high dimensional contingency table.

More formally, we have a linear map 
$$\pi_{\Delta}: \mathbb{R}^r  \longrightarrow  \mathbb{R}^d$$
which is the linear map that computes (some collection $\Delta$ of) marginals of a $r_1 \times r_2 \times \ldots \times r_n$ $n$-way array.  The fundamental problem is to characterize the image semigroup
$$ \mathcal{S}_\Delta := \pi_\Delta( \mathbb{N}^r).$$
Characterizing the semigroup is a problem that typically falls into two pieces.  One piece is to characterize the image cone
$$ \mathcal{C}_\Delta := \pi_\Delta( \mathbb{R}_{\geq 0}^r),$$
by giving its facet defining inequalities.  The second piece concerns understanding the discrepancy between the semigroup $\mathcal{S}_\Delta$ and its normalization; that is, describing the set of holes:
$$
\mathcal{H}_\Delta := (\mathcal{C}_\Delta \cap \mathbb{Z}^d) \setminus \mathcal{S}_\Delta.
$$
As a very special case of this second problem one is lead to the question:

\begin{ques}\label{ques:normal}
For which $\Delta$ and $r$ is the set of holes $\mathcal{H}_\Delta$ empty?  In other words, for which $\Delta$ and $r$ is the semigroup $\mathcal{S}_\Delta$ \emph{normal}?
\end{ques}

A general answer to question \ref{ques:normal} seems out of reach with present techniques.  Even in the case of a three cycle $\Delta = [12][13][23]$ it is still open to classify the $r$ which produce a normal semigroup   (though there remain only a finite number of cases left to check at present).

In this paper, we focus on the special case where $G$ is a graph, and $r_1 = \ldots = r_n =2$, the so-called \emph{binary graph models} \cite{Develin2003}.  In particular, we will show the following:

\begin{thm}\label{thm:main}
Let $G$ be a graph with vertex set $[n]$ and let $r_1 = \ldots = r_n = 2$.  Then the marginal semigroup $\mathcal{S}_G$ is normal if and only if $G$ is free of $K_4$ minors.
\end{thm}

Note, in particular, that Theorem \ref{thm:main} implies that marginal semigroups are rarely normal.
Section 2 contains the proof of Theorem \ref{thm:main}, whose main idea is to relate normality properties of graphs to normality properties of subgraphs.  One of the key take-home messages of this proof is that these normality problems can often be addressed by taking the geometry of the cone $\mathcal{C}_\Delta$ into account, instead of working only with the semigroup.  This is the content of Lemma \ref{lem:facepopper}.  Section 3 is devoted to a description of some further possible directions of exploration.

%%%%
%%%%
%%%%
  
\section{The Proof}

In this section we formally set up the notion of the marginal cone and the marginal semigroup.  Then we prove some key lemmas for the general normality question for marginal semigroups, and remind the reader of some useful structural results in graph theory on $K_4$ minor-free graphs.  These ideas come together to provide the proof of Theorem \ref{thm:main}.

Fix a positive integer $n$.  Let $[n] = \{1,2,\ldots, n\}$.  For each $i \in [n]$ let $r_i$ be a positive integer.   For any set $F \subseteq [n]$ let $\mathcal{R}_F = \prod_{i \in F} [r_i]$ be a set of indices.  In the special case $F = [n]$, let $\mathcal{R} := \mathcal{R}_F$.  For any set $A$, let $\mathbb{R}^A$ be the real vector space of dimension $\#$ with basis $\{e_a : a \in A\}$.  In the special case where $A = \mathcal{R}$ we say that
$\mathbb{R}^\mathcal{R}$ is the space of $r_1 \times r_2 \times \cdots \times r_n$ tables.  Elements of $\mathcal{R}$ are $\bfi = (i_1, \ldots, i_n)$.  If $F \subseteq [n]$, we denote $\bfi_F = (i_f)_{f \in F}$.

Let $\Delta \subseteq 2^{[n]}$ be a collection of subsets of $n$.     We define the \emph{marginal map} $\pi_\Delta$ by the formula:
$$\pi_\Delta : \mathbb{R}^\mathcal{R} \longrightarrow  \bigoplus_{F \in \Delta}  \mathbb{R}^{\mathcal{R}_F}; \quad  e_{\bfi} \mapsto  \oplus_{F \in \Delta}  e_{\bfi_F}$$
on unit vectors and extending the map linearly to arbitrary elements of $\mathbb{R}^\mathcal{R}$.  Note that the ``marginal of a marginal is a marginal'', so that will assume that if $F \in \Delta$ and $S \subseteq F$, then $S \in \Delta$ as well.  Thus, we will refer to $\Delta$ as a simplicial complex.  The coordinates on $\bigoplus_{F \in \Delta}  \mathbb{R}^{\mathcal{R}_F}$ are denoted $p^F_{i_F}$, where $F \in \Delta$ and $i_F \in \mathcal{R}_F$.

The \emph{marginal cone} $\mathcal{C}_\Delta$ is the image of the nonnegative orthant under the marginal and the marginal semigroup $\mathcal{S}_\Delta$ is the image of the lattice points in the nonnegative orthant under the marginal map
$$\mathcal{C}_\Delta :=  \pi_\Delta ( \mathbb{R}^\mathcal{R}_{\geq 0}),  \quad \mathcal{S}_\Delta :=  \pi_\Delta ( \mathbb{N}^\mathcal{R}).$$
Clearly both $\mathcal{C}_\Delta$ and $\mathcal{S}_\Delta$ depend on both $\Delta$ and $\mathcal{R}$, but we suppress the dependence on $\mathcal{R}$ in the notation.

First, we consider how two general operations on a simplicial complex $\Delta$ relate to the geometry of the marginal cone and the marginal semigroup.  The first operation is edge contraction.  Suppose the $L \subseteq [n]$ are a set of vertices $L  \in \Delta$.  Define the edge contraction by $\Delta / L$, on $([n] \cup \{v\}) \setminus L$ by
$$\Delta / L := \{ S  \in \Delta : S \cap L = \emptyset \} \cup  \{ S \cup \{v\} : S \cap L \neq \emptyset \}.$$
When we contract the edge, we set $r_v =  \min_{f \in F} r_f$.
The second operation is vertex deletion.  If $v \in [n]$ define the vertex deletion by
$$\Delta \setminus v  =  \{ S \in \Delta :  v \notin S \}.$$
The following lemma seems to be known in the literature on marginal polytopes, though it is difficult to find a precise reference.

\begin{lemma}\label{lem:minors}
Suppose that $\Gamma$ is obtained from $\Delta$ by either 
\begin{enumerate}
\item deleting a vertex, or
\item contracting an edge.
\end{enumerate}
Then $\mathcal{C}_\Gamma$ is (isomorphic to) a face of $\mathcal{C}_\Delta$ and $\mathcal{S}_\Gamma$ is isomorphic to  $\mathcal{S}_\Delta \cap \mathcal{C}_\Gamma$.  In particular, if $\mathcal{S}_\Delta$ is normal then so is $\mathcal{S}_\Gamma$.
\end{lemma}

\begin{proof}
Faces of a cone (or semigroup) are obtained by taking the intersection of the cone (or semigroup) with a hyperplane of the form $c^T p = 0$, where $c^T p\geq 0$ is a valid inequality on the cone (or semigroup).  To prove the lemma, it suffices to find such a hyperplane such that the resulting cone (or semigroup) is isomorphic to the cone (or semigroup) for the corresponding simplicial complex.  This will imply the statement on normality, because any facial semigroup of a normal semigroup is normal.
 
For the case of deleting a vertex $v$, consider the hyperplane given by
$$
c^Tp = \sum_{j = 2}^{r_v} p^{\{v\}}_j.
$$
Clearly $c^Tp \geq 0$ is a valid inequality on $\mathcal{C}_\Delta$, so $\mathcal{C}_\Delta \cap \{p : c^T p  = 0 \}$ is a face of $\mathcal{C}_\Delta$.  Furthermore, both it, and the corresponding semigroup are generated by $\pi_\Delta({e_{i_1, \ldots, i_n}})$ such that $i_v = 1$.  So this is the same as the marginal cone (or semigroup) with the same $\Delta$ and $r_v = 1$.  However, in this case, if $v \subset F$, then the $F$ marginal of a table in this face is the same as the $F \setminus \{v\}$ marginal.  Hence, this facial cone (or semigroup) is isomorphic to $\mathcal{C}_{\Delta \setminus v}$ (or $\mathcal{S}_{\Delta \setminus v }$.

For the case of contracting an edge $L$, we can take the hyperplane
$$
c^T p =  \sum_{i_L \in \mathcal{R}_L \setminus D}  p^L_{i_L}
$$
where $D$ is the diagonal $D = \{i_L \in \mathcal{R}_F : i_{l_1} = i_{l_2} \cdots \}$.  Clearly $c^Tp \geq 0$ is a valid inequality on $\mathcal{C}_\Delta$, so $\mathcal{C}_\Delta \cap \{p : c^T p  = 0 \}$ is a face of $\mathcal{C}_\Delta$.  Furthermore, both it, and the corresponding semigroup are generated by $\pi_\Delta({e_{i_1, \ldots, i_n}})$ such that $i_{l_1} = i_{l_2} = \cdots$.  Then if $F \cap L$ is nonempty, then the $F$ marginal of a table on this face, can be recovered from the marginal of $F \setminus L \cup {l_1}$ where $l_1$ is any element of $L$.  Since we can take the same $l_l$ for all such $F$, we deduce that this facial cone (or semigroup) is isomorphic to $\mathcal{C}_{\Delta / L}$ (or $\mathcal{C}_{\Delta / L}$).
\end{proof}

The holy grail for studying the geometry of the marginal cone would be a result about removing an element of $\Delta$.  It is unlikely that there is a very general result that removing elements from $\Delta$ preserves normality.  Our next crucial result, Lemma \ref{lem:facepopper} concerns a special case of when normality is preserved on removing a face.

To explain Lemma \ref{lem:facepopper} we first need to introduce a slightly modified version of our coordinate system for speaking of marginal cones a semigroups, which allows us to work with full dimensional cones and semigroups.    To reduce the dimensionality, we show that we only need to consider those $i_F$ that do not contain $r_f$, for all $f \in F$.  Note however that we always include a coordinate $p^\emptyset$, which gives the sample size of a table $\mathbf{u}$.

To this end let $R_F  =  \prod_{f \in F}  [r_f - 1]$.

\begin{prop}\label{prop:newcoord}
Consider the map from $L:\bigoplus_{F \in \Delta}  \mathbb{R}^{\mathcal{R}_F} \to \bigoplus_{F \in \Delta}  \mathbb{R}^{R_F}$ that deletes all the $p^F_{i_F}$ such that there is an $f \in F$ with $i_f = r_f$.  Then there is a linear map $M: \bigoplus_{F \in \Delta}  \mathbb{R}^{R_F} \to \bigoplus_{F \in \Delta}  \mathbb{R}^{\mathcal{R}_F}$ such that $M \circ L( \mathcal{C}_\Delta) = \mathcal{C}_\Delta$.  Furthermore, $L( \mathcal{C}_\Delta)$ is a full dimensional polyhedral cone.
\end{prop}

\begin{proof}
The full-dimensionality will follow from the existence of the inverse map $M$, because the dimension of the space $\bigoplus_{F \in \Delta}  \mathbb{R}^{R_F}$ equals the dimension of the marginal cone by Corollary 2.7 in \cite{Hosten2002}.

To prove the existence of the inverse map, we use a familiar M\"obius inversion style formula.  In particular, we need to show that we can recover $p^F_{i_F}$ where $i_F$ contains some zeros from all the $p^F_{i_F}$ where these last contain no zeroes.  It suffices to show this in the case where $i_F$ is the vector of all elements equal to $r_f$, we we denote $r_F$.  For each $S \in \Delta$ let $q^F =  \sum_{i_F \in R_F} p^F_{i_F}$.  Then we have
$$p^F_{r_F}  =  \sum_{S \subseteq F} (-1)^{\#S} q^S.$$
\end{proof}

\begin{lemma}\label{lem:facepopper}
Suppose that $\Delta$ and $r_1, \ldots, r_n$ are such that $\mathcal{S}_\Delta$ is normal.  Let $F$ be a maximal face of $\Delta$.   Let $A, B$ be integral matrices such that 
$$\mathcal{C}_\Delta = \left\{(x,y) =  (p^S_{i_S}: S \in \Delta \setminus F,   p^F_{i_F})  :  (A \, \, B) {x \choose y}  \geq 0  \right\}.$$
If the matrix $B$ satisfies the property that the system $By \geq b$ has an integral solution for all $b$ such that the system has a real solution, then $\mathcal{S}_{\Delta \setminus F}$ is normal.
\end{lemma}

\begin{proof}
Note that this lemma is merely a general property of projections of cones $(A \, \, B) {x \choose y}  \geq 0$ onto one set of coordinates.  Indeed, consider the cone $\mathcal{C}_{\Delta \setminus F}$, and let $x$ be an integral point in it.  If we can find an integral $y$ such that $(x,y) \in \mathcal{C}_{\Delta}$, we will be done because $\mathcal{C}_\Delta$ is normal since $(x,y) \in \mathcal{S}_\Delta$ implies $x \in \mathcal{S}_{\Delta \setminus F}$.  The condition on $B$ guarantees that an integral $y$ always exists, since the set of such $y$ is the solution to the system $By  \leq - Ax$ where $-Ax$ is integral and real feasible.
\end{proof}

One more tool we will need for proving normality, is that normality is preserved when gluing two simplicial complexes together according to a reducible decomposition.  A simplicial complex $\Delta$ has a \emph{reducible decomposition} $(\Delta_1, S, \Delta_2)$ is $\Delta = \Delta_1 \cup \Delta_2$, $\Delta_1 \cap \Delta_2 = 2^{S}$ (where $2^S$ is the power set of $S$), and either $\Delta_1$ nor $\Delta_2 = 2^S$.  A simplicial complex with a reducible decomposition is called \emph{reducible}.  A simplicial complex is \emph{decomposable} if it is either a simplex (of the form $2^K$) or it is reducible and both $\Delta_1$ and $\Delta_2$ are decomposable.  

\begin{lemma}
Let $\Delta$ be a reducible simplicial complex with decomposition $(\Delta_1, S, \Delta_2)$.  If $\mathcal{S}_{\Delta_1}$ and $\mathcal{S}_{\Delta_2}$ are normal then so is $\mathcal{S}_\Delta$.
\end{lemma}

\begin{proof}
Let $x = (x_{\Delta_1}, x_{\Delta_2})  \in \mathcal{C}_\Delta$.  (Without changing the issues of normality where can repeat all the elements of $x^T_{i}$ such that $T \subseteq S$).  We must show that $x \in \mathcal{S}_\Delta$.  Now, both $\mathcal{S}_{\Delta_1}$ and $\mathcal{S}_{\Delta_2}$ are normal, which implies that $x_{\Delta_1} \in \mathcal{S}_{\Delta_1}$ and $x_{\Delta_2} \in \mathcal{S}_{\Delta_2}$.  Let $|\Delta_i| =  \cup_{T \in \Delta_i} T$ be the ground set of $\Delta_i$.  Let $\Gamma$ be the simplicial complex with facets $|\Delta_1|$ and $ |\Delta_2|$.

For each $i = 1,2$, by normality  of $\mathcal{S}_{\Delta_i}$, there exist integral vectors $y_i \in \mathcal{S}_{|\Delta_i|}$ such that $\pi_{\Delta_i}(y_i) = x_i$.  Furthermore, since $\pi_S(x_1) = \pi_S(x_2)$ we have that the pair $(y_1, y_2)$ satisfies the natural equality relations to lie in the cone $\mathcal{C}_\Gamma$.  However, $\Gamma$ is an example of decomposable simplicial complex, so its only facet defining inequalities come from positivity \cite{Geiger2006}, so $(y_1, y_2)$ lies in $\mathcal{C}_R$.  It is also known that decomposable models are normal \cite{Hosten2002}, so $(y_1, y_2) \in \mathcal{S}_R$.  This implies that $(\pi_{\Delta_1}(y_1), \pi_{\Delta_2}(y_2) ) = (x_{\Delta_1}, x_{\Delta_2}) \in \mathcal{S}_\Delta$, so $\mathcal{S}_\Delta$ is normal.
\end{proof}

To prove Theorem \ref{thm:main} we need some important folklore decompositions for $K_4$ minor-free graphs.  Recall that a graph $G$ is \emph{chordal} if every cycle of length $\geq 4$ has a diagonal.  Chordal graphs have an equivalent characterization, that they can be built up inductively starting with complete graphs, by gluing two chordal graphs together along a complete subgraph.  A chordal triangulation of a graph $G$ is a chordal graph $H$ such that $G \subseteq H$.  The \emph{tree width} of $G$, denoted $\tau(G)$ is one less than the minimal clique number over all chordal triangulations of $G$.  Note that chordal graphs are closely related to decomposable simplicial complexes: a simplicial complex is decomposable if and only if it is the complex of cliques of a chordal graph.  A folklore result relates tree width and $K_4$ minor free graphs.

\begin{thm}\label{thm:treewidth}
A graph $G$ is free of $K_4$ minors if and only if $\tau(G) \leq 2$.
\end{thm}

The following lemma, which can be handled computationally using the program Normaliz \cite{normaliz}.

\begin{lemma}\label{lem:normnot}  Let $r_i = 2$ for all $i$.  Then the semigroup $\mathcal{S}_{K_l}$ for the complete graph $K_l$ is normal 
if and only if $l = 1,2,3$.
\end{lemma}

\begin{proof}
For $l =1, 2$, $\mathcal{C}_{K_l}$ is a unimodular simplicial cone and normality follows.  For $l = 3$, the semigroup $\mathcal{S}_{K_3}$ is 7 dimensional with 8 generators and normality is verified with Normaliz.  For $l = 4$, a computation with Normaliz shows that the set $(\mathcal{C}_{K_4} \cap \bigoplus_{F \in K_4} \mathbb{Z}^{\mathcal{R}_F} ) \setminus \mathcal{S}_{K_4}$ consists of a single point (see also, \cite{Takemura2008}).  For $l \geq 5$, $\mathcal{S}_{K_l}$ is not normal by Lemma \ref{lem:minors}.
\end{proof}

Now for our last piece of the argument, we need to recall a result about the polyhedral structure of cone $\mathcal{C}_G$ in the case the $G$ is $K_4$ minor free.  We work with the alternate coordinate system introduced in Proposition \ref{prop:newcoord}.  Since we are working will that case that $r_i = 2$, each coordinate $p^F_{i_F}$ will have $i_F = (1,1,\ldots)$, a vector of all ones.  To simplify notation, we merely use $p^F$ to denote this coordinate.

\begin{thm}\cite{Baharona1986} \label{thm:facetsbinary}
Let $r_i = 2$ for all $i$.  If the graph $G$ is free of $K_4$ minors, the cone $\mathcal{C}_G$ is the solution to the following system of inequalities:
$$ p^{jk} \geq 0, p^{j} - p^{jk} \geq 0, p^{k} - p^{jk} \geq 0, p^{\emptyset} - p^{j} - p^k + p^{jk} \geq 0   \mbox{ for all } jk \in E$$
$$
\sum_{jk \in O}  p^{jk}  - \sum_{jk \in C \setminus O} p^{jk} - \sum_{j \in V(O)} p^j  + \sum_{j \in V(C \setminus O)} p^j + \tfrac{ \#O - 1}{2} p^\emptyset  \geq 0$$  $$\mbox{ for all cycles } C \in G \mbox{ and odd subsets } O \subseteq C
$$
where $V(O)$ and $V(C \setminus O)$ denotes the set of vertices that appear in $O$ and $C \setminus O$, respectively.
\end{thm}

We now have all the tools in hand to prove our main results on normality.

\begin{proof}[Proof of Theorem \ref{thm:main}]
First of all, we will show that if a graph $G$ has a $K_4$ minor, then $\mathcal{S}_G$ could not be normal.  If a graph has a $K_4$ minor, then that minor can be realized by vertex deletions and edge contractions alone.  Since Lemma \ref{lem:normnot} implies that $\mathcal{S}_{K_4}$ is not normal, Lemma \ref{lem:minors} implies that $\mathcal{S}_G$ is not normal.

Now suppose that $G$ is free of $K_4$ minors.  By Theorem \ref{thm:treewidth}, $G$ has tree width $\leq 2$.  If $G$ is a chordal graph with $\tau(G) \leq 2$ it can be broken down into cliques of size $1$, $2$ and $3$ by reducible decompositions along cliques of sizes $0$, $1$, or $2$.  Because cliques of size $0$, $1$, and $2$ are faces of the simplicial complex of $G$, these reducible decompositions preserve normality.    Lemma \ref{lem:normnot} implies that $\mathcal{S}_{K_l}$ is normal if $l =1, 2,$ or $3$.  Thus, $\mathcal{S}_G$ is normal if $G$ is a chordal graph.  Since every $K_4$ minor free graph is the edge subgraph of a chordal $K_4$ minor free graph, it suffices to show that normality is preserved when deleting an edge from a $K_4$ minor free graph.

To prove this last claim, let $G$ be a $K_4$ minor free graph such that $\mathcal{S}_G$ is normal, and let $e$ be an arbitrary edge in $G$, and let $H = G \setminus e$, the edge deletion of $G$.  According to Theorem \ref{thm:facetsbinary}, the cone $\mathcal{C}_G$ is the solution to the system of equations that have the form $(A,B) (x, p^e)^T \geq 0$, where all the entries of $B$ are either $\pm 1, 0$.  In particular, because there is only one coordinate $p^e$, $B$ is a column of $0, \pm 1$.    Thus, the system $B y \geq b$, where $b$ is an integral vector such that   $By \geq b$ has a real solution is equivalent to a system $c_1 \leq y \leq c_2$, where $c_1$ and $c_2$ are integers.  Every such system which has a real solution has an integral solution.  Lemma \ref{lem:facepopper} implies that $\mathcal{S}_H$ is normal.
\end{proof}

%%%%%%%%%%%%%%%%%%%%%%%%%%%%%%%%%%%%%%
%%%%%%%%%%%%%%%%%%%%%%%%%%%%%%%%%%%%%%
%%%%%%%%%%%%%%%%%%%%%%%%%%%%%%%%%%%%%%
%%%%%%%%%%%%%%%%%%%%%%%%%%%%%%%%%%%%%%

\section{Further Directions}

This paper has solved the normality question for binary graph models, however, this is still very far from a complete solution to the normality problem for arbitrary sets of marginals for arbitrary sized tables.    In this section, we outline some possible directions for further research.

First of all, the technique used in this paper (gluing along facets and then removing a facet preserving normality, together with knowledge of the defining inequalities of structure of marginal cone) can  be used to prove normality in many more situations besides just for all $r_i = 2$. 

\begin{ex}
Consider the three cycle $\Delta = [12][13][23]$, with $r_1 = 2, r_2 = 4,$ and  $r_3 = 3$.  A direct computation with Normaliz shows that $\mathcal{S}_\Delta$ is normal in this case, and the facet description computed there has that, for the edge $13$, the corresponding $B$ matrix is a $110 \times 2$ matrix where each row is one of the vectors 
$$(0,0), \pm(0,1), \pm (1,0) \pm(1,1).$$
It is easy to see that such a $B$ satisfies the conditions of Lemma \ref{lem:facepopper}, so that we can delete the edge $[13]$ to preserve normality.  Of course, this is not very interesting as this produces a decomposable complex $[12][23]$, which are always normal.

Now suppose that we have the complex $\Delta = [12][13][14][23][34]$, with 
$r_1 = 2, r_2 = 4, r_3 = 3, $ and $r_4 = 4$.  This is reducible into two models of the above type, along the edge $[13]$, so $\mathcal{S}_\Delta$ is normal.  Since $\Delta$ is reducible, the polyhedral structure of $\mathcal{C}_\Delta$ is obtained by taking the union of the two constraint sets for the two halves.  In particular, if we try to delete the edge $[13]$, we get a $B$ matrix of size $220 \times 2$, where each row is one of the vectors
$$(0,0), \pm(0,1), \pm (1,0) \pm(1,1).$$
Hence the resulting four cycle  $\Delta' = [12][14][23][34]$ is normal with  $r_1 = 2, r_2= 4, r_3= 3,$ and $ r_4= 4$.   \qed
\end{ex}

The preceding example illustrates that the techniques can be applied to more general graphs.  It seems natural to hope that once the (non)normality of all the three cycle models $[12][13][23]$ have been determined for all $r_1, r_2,$ and $r_3$, that a combination of techniques from the preceding Section could be used to decide normality for arbitrary graphs.  

A second situation where these techniques are likely to apply fruitfully is for arbitrary simplicial complexes but with all $r_i = 2$.  This is because the condition of Lemma \ref{lem:facepopper} is especially easy to verify in this case, because $B$ is always a column vector.  In fact, in examples we have investigated, the conditions of Lemma \ref{lem:facepopper} seem to be necessary and sufficient for guaranteeing normality on removing a maximal face.

\begin{ex}
Let $\Delta = [12][134][234]$, and $r_1 = r_2 = r_3 = r_4 = 2$.  A direct computation in Normaliz shows that $\mathcal{S}_\Delta$ is normal.  The facet defining inequalities of $\mathcal{C}_\Delta$ have $0, \pm 1, \pm 2$ coefficients on the coordinate $p^{134}$.  Removing facet $[134]$ produces the complex $\Gamma = [12][13][14][234]$, and Normaliz verifies that $\mathcal{S}_\Gamma$ is not normal.  \qed
\end{ex}

\begin{ques}
Let $r_i = 2$ for all $i$.
Suppose that $\Delta$ is such that $\mathcal{S}_\Delta$ is normal.  Let $F$ be a maximal face of $\Delta$.   Let $A, B$ be integral matrices such that 
$$\mathcal{C}_\Delta = \left\{(x,y) =  (p^S_{i_S}: S \in \Delta \setminus F,   p^F_{i_F})  :  (A \, \, B) {x \choose y}  \geq 0  \right\},$$
with irredundant and minimal description.
If the matrix $B$ has an entry of absolute value $> 1$ does this imply that $\mathcal{S}_{\Delta \setminus F}$ is not normal?
\end{ques}

Note, however, that Lemma \ref{lem:facepopper}, while adequate for the examples we have encountered here, is probably not the best possible result along these lines.  This is because the set of $b$ on which the condition ``real solution implies integral solution'' needs to be valid is, in fact, very small.

\begin{pr}
Find a stronger version of Lemma \ref{lem:facepopper}.
\end{pr}  

Lastly we would like to address what we think is the main take-away message of this paper.  This is that the polyhedral structure of the cone $\mathcal{C}_\Delta$ seems crucial to studying the normality of $\mathcal{S}_\Delta$.  Unfortunately, the only class where a nice polyhedral description of the marginal cone $\mathcal{C}_\Delta$ is known is the case of $K_4$-minor free graphs with all $r_i = 2$. 

\begin{pr}
Find new classes of $\Delta$ where there is a elegant uniform description of the polyhedral cone $\mathcal{S}_\Delta$.
\end{pr}

A related situation where there is an elegant polyhedral description concerns the cut cones \cite{Deza1997}.  We hope that these properties could also be used to resolve the normality questions for cut cones from \cite{Sturmfels2008}.

\section*{Acknowledgments}
The author received support from the U.S. National Science Foundation under grant DMS-0840795.

%%%%%%%%%%%%%%%%%%%%%%%%%%%%%%%%% 
%%%%%%%%%%%%%%%%%%%%%%%%%%%%%%%%% 
%%%%%%%%%%%%%%%%%%%%%%%%%%%%%%%%% 
%%%%%%%%%%%%%%%%%%%%%%%%%%%%%%%%% 

\end{document}